\documentclass[letterpaper,10pt]{amsart}
\usepackage{amssymb}
\usepackage{amsthm}
\usepackage[utf8]{inputenc}
\usepackage[margin=1.0in]{geometry}
\usepackage{pst-node}
\usepackage{tikz-cd} 
\usepackage{thmtools}
\usepackage{thm-restate}
\usepackage{hyperref}
\usepackage{amsrefs}
\usepackage[noabbrev,capitalize]{cleveref}
\usepackage{amsmath}
\usepackage{amsfonts}
\usepackage{tikz-cd}
\usepackage{mathtools}
\usepackage{bbm}
\usepackage{hyperref}
\usepackage{blindtext, xcolor}
\usepackage{comment}
\usepackage{romannum}
\usepackage{multirow}
\usepackage{array}
\usepackage{tensor}

\DeclareMathOperator{\Cl}{Cl}

\DeclareMathOperator{\Sym}{Sym}

\DeclareMathOperator{\Disc}{Disc}

\DeclareMathOperator{\St}{St}
\DeclareMathOperator{\Nm}{Nm}

\newcommand{\ZZ}{\mathbb{Z}}      
\newcommand{\QQ}{\mathbb{Q}}      
\newcommand{\cO}{\mathcal{O}}      
\newcommand{\CC}{\mathbb{C}}      

\newcommand{\mfp}{\mathfrak{p}}    
\newcommand{\mfa}{\mathfrak{a}}

\newtheorem{theorem}{Theorem}[section]

\newtheorem{lemma}[theorem]{Lemma}

\theoremstyle{definition}

\theoremstyle{remark}

\usepackage[english]{babel}
\usepackage[utf8]{inputenc}
\usepackage{wrapfig}
\usepackage{caption}
\title{The Steinitz Realization Problem}

\author{Sameera Vemulapalli}
\address{Department of Mathematics, Harvard University}
\email{vemulapalli@math.harvard.edu}

\subjclass[2010]{11R04 (primary), 11R09, 11R29, 14H05 (secondary)}
\keywords{Algebraic numbers; Polynomials; Class groups; Geometry of numbers}

\usepackage[english]{babel}
\usepackage[utf8]{inputenc}

\begin{document}

\maketitle
\pagenumbering{arabic}

\begin{abstract}
Let $K$ be a number field and let $n \in \ZZ_{>1}$. The \emph{Steinitz realization problem} asks: does every element of $\Cl(K)$ occur as the Steinitz class of a degree $n$ extension of $K$? In this article, we give an affirmative answer to the Steinitz realization problem for all $n$ and $K$.
\end{abstract}

\section{Introduction}

Fix $n \in \ZZ_{> 1}$ a positive integer. To any degree $n$ extension of number fields $L/K$, we may canonically associate an element of the ideal class group of $K$, namely $[\mfa]$ where $\cO_L \simeq \cO_K^{n-1} \oplus \mfa$ as an $\cO_K$-module. The class $[\mfa]$ is called the \emph{Steinitz class} of $L/K$ and is denoted $\St(L/K)$. By a theorem of Hecke, the square of the Steinitz class $\St(L/K)$ is the class of the relative discriminant $\Disc(L/K)$. The \emph{Steinitz realization problem} asks: does every element of $\Cl(K)$ occur as the Steinitz class of a degree $n$ extension of $K$? 

When $n = 2,3,4,5$, there is an affirmative answer to the Steinitz realization problem; in these cases the Steinitz classes of $S_n$-extensions are \emph{equidistributed} in $\Cl(K)$ when degree $n$ number fields are ordered by discriminant. The case $n = 2,3$ was proven by Kable and Wright \cite{kab} and the case $n = 4,5$ was proven by Bhargava, Shankar, and Wang \cite{bsw}. The proofs rely on the parametrizations of degree $n$ number fields; for $n \geq 6$, there are no such parametrizations. The purpose of this article is to give an affirmative answer to the Steinitz realization problem for all $n$.

\begin{theorem}
\label{mainthm}
Every element of $\Cl(K)$ occurs as the Steinitz class of a degree $n$ extension of $K$ with squarefree discriminant.
\end{theorem}

We give a sketch of the proof. A \emph{rank $n$ ring} over $\cO_K$ is a ring that is locally free of rank $n$ as an $\cO_K$-module. By work of Wood \cite{wood}, a binary $n$-ic form 
\[
	f(x,y) = f_0 x^n + \dots + f_n y^n \in \Sym^n(\cO_K \oplus \cO_K)
\]
gives rise to a rank $n$ ring $R_f$. If $\mfa$ is an integral ideal of $\cO_K$ such that $f_{n-1} \in \mfa$ and $f_n \in \mfa^2$, we construct a rank $n$ ring $R_f(\mathbf{a})$ and an inclusion $R_f \xhookrightarrow{} R_f(\mathbf{a})$. When $f$ is irreducible, both rings are orders in a number field. The Steinitz class of $R_f(\mfa)$ is $[\mfa^{-1}]$ and $\Disc(R_f(\mathbf{a})) = \mfa^{-2}\Disc(R_f)$.

If $\Disc(R_f(\mfa))$ is squarefree as an ideal and $f(x,y)$ is irreducible, then $R_f(\mfa)$ is the maximal order in a degree $n$ extension of $K$ with Steinitz class $[\mfa^{-1}]$. Thus, \cref{mainthm} immediately follows from the existence of an appropriate $f(x,y)$.
 
\begin{restatable}{theorem}{sievethm}
\label{sievethm}
For every nonzero proper prime ideal $\mfa$ of $\cO_K$ coprime to $(n!)$ and subject to the condition that $\#(\cO_K/\mfa) > n^{n-2}$, there are infinitely many monic degree $n$ polynomials
\[
	f(x) \coloneqq x^n + f_1 x^{n-1} \dots + f_n
\]
with coefficients in $\cO_K$ such that:
\begin{enumerate}
\item $f$ is irreducible over $K$;
\item $f_{n-1} \in \mfa$ and $f_n \in \mfa^2$;
\item and $\Disc(R_f(\mfa))$ is squarefree as an ideal in $K$.
\end{enumerate}
\end{restatable}

The proof of \cref{sievethm} is a slight generalization of the proof given by Kedlaya when proving that there exist infinitely many irreducible monic polynomials over $\ZZ$ with squarefree discriminant \cite{ked}. 

\subsection{Acknowledgements}
The author would like to thank Ravi Vakil and Kiran Kedlaya for valuable conversations. The author was supported by the NSF under grant number DMS2303211. Additionally, the author would like to thank the referee for their thorough and helpful comments on an earlier version on this manuscript.

\section{Construction of $R_f(\mfa)$}
By work of Wood \cite{wood}, a binary $n$-ic form $f_0 x^n + \dots + f_n y^n \in \Sym^n(\cO_K \oplus \cO_K)$ gives rise to a rank $n$ ring $R_f$ as follows. As an $\cO_K$-module, put $R_f \coloneqq \cO_K^{n}$ and label the copies of $\cO_K$ so we write:
\[
	R_f = (\cO_K)_0 \oplus \dots \oplus (\cO_K)_{n-1}.
\]
Put $u_1,\dots,u_{n-1} \coloneqq 1$ and $u_n \coloneqq -f_n$. Define $1 \in (\cO_K)_0$ to be the multiplicative identity in $R_f$. For $1 \leq i \leq j \leq n-1$ and $1 \leq k \leq n$, define the multiplication maps $(\cO_K)_i \otimes (\cO_K)_j \rightarrow (\cO_K)_k$ by:
\begin{enumerate}
\item $x \otimes y \rightarrow -f_{i + j - k}u_k xy$ if $\max(i + j - n, 1) \leq k \leq i$;
\item $x \otimes y \rightarrow f_{i + j - k}u_k xy$ if $j < k \leq \min(i + j, n)$;
\item and $x \otimes y \rightarrow 0$ otherwise.
\end{enumerate}
In the notation above, $(\cO_K)_n$ refers to $(\cO_K)_0$.

For any nonzero integral ideal $\mfa$ of $K$, define $R_f(\mfa)$ to be the locally free $\cO_K$-module given by:
\[
	R_f(\mfa) \coloneqq (\cO_K)_0 \oplus \dots \oplus \mfa^{-1}(\cO_K)_{n-1}.
\]
There is a natural inclusion of $\cO_K$-modules given by $R_f \subseteq R_f(\mfa)$.

\begin{lemma}
$R_f(\mfa)$ has an $\cO_K$-algebra structure extending that of $R_f$ if and only if $f_{n-1} \in \mfa$ and $f_n \in \mfa^2$. When $f_{n-1} \in \mfa$ and $f_n \in \mfa^2$, there is a unique $\cO_K$-algebra structure on $R_f(\mfa)$ extending that of $R_f$, the Steinitz class of $R_f(\mfa)$ is $[\mfa^{-1}]$ and $\Disc(R_f(\mfa)) = \mfa^{-2}\Disc(R_f)$.
\end{lemma}
\begin{proof}
The first statement follows directly from the multiplication table. The second is immediate from the construction.
\end{proof}

\section{Proof of \cref{sievethm}}
The proof given in this section is essentially that given by Kedlaya \cite{ked}. The main differences are that we work over a number field and impose a few additional congruence conditions needed for our application. We repeat much of Kedlaya's construction verbatim, and encourage the reader to refer to Kedlaya's paper for more detail. Fix a class $[\mfa] \in \Cl(K)$, and pick a representative prime ideal $\mfa$ coprime to $(n!)$. For $a_1,a_2,\dots,a_{n-1},b \in \cO_K$, put:
\[
	Q_{\vec{a}}(x) \coloneqq n(x - a_1/n)(x - a_2)\dots (x - a_{n-1})
\]
\[
	P_{\vec{a},b}(x) \coloneqq  b + \int_{0}^x Q_{\vec{a}}(t) \, dt
\]
The integral in the definition of $P_{\vec{a},b}(x)$ should be interpreted formally; $Q_{\vec{a}}(x)$ is a polynomial. The principal ideal generated by the discriminant of $P_{\vec{a},b}$ is denoted $\Delta_{\vec{a},b}$ and is equal to the resultant of $P_{\vec{a},b}(x)$ and $Q_{\vec{a}}(x)$. So we obtain the following equality of ideals:
\[
	\Delta_{\vec{a},b} = (n^n P_{\vec{a},0}(a_1/n) + n^nb)\prod_{i = 2}^{n-1}(P_{\vec{a},0}(a_i) + b).
\]

Suppose $a_1 \in \mfa$ and $b \in \mfa^2$. If we write $P_{\vec{a},b}(x) = x^n + f_1x^{n-1} + \dots + f_n$ and assume that $P_{\vec{a},b}(x)$ has coefficients in $\cO_K$, then $f_n = b \in \mfa^2$ and
\[
	f_{n-1} = (-1)^{n-1} \prod_{i = 1}^{n-1} a_i \in \mfa.
\]
Furthermore, $\mfa^2$ divides $(n^n P_{\vec{a},0}(a_1/n) + n^nb)$ and hence divides $\Delta_{\vec{a},b}$. Therefore, we wish to show that there exists a choice of $a_1,\dots,a_{n-1},b$ with  $a_1 \in \mfa$ and $b \in \mfa^2$ such that $P_{\vec{a},b}(x)$ has integral coefficients, is irreducible, and the ideal $\mfa^{-2}\Delta_{\vec{a},b}$ is squarefree. We will do this by fixing an appropriate choice of $a_1,\dots,a_{n-1}$ and then applying a squarefree sieve to choose $b$.

We will first show that there exist $a_1,\dots,a_{n-1} \in \cO_K$ such that:
\begin{enumerate}
\item $a_1 \in \mfa$ and $a_i \equiv -1 \pmod{\mfa}$ for $i = 2,\dots,n-1$;
\item $P_{\vec{a},0}$ has integral coefficients;
\item for every prime ideal $\mfp \neq \mfa$, there exists $b$ such that $\Delta_{\vec{a},b}$ is not divisible by $\mfp^2$;
\item there exists $b_2 \in \mfa^2$ such that $\Delta_{\vec{a},b_2}$ is not divisible by $\mfa^3$;
\item there exists a prime ideal $\mfp_1$ and $b_1 \in \cO_K$ such that $P_{\vec{a},b_1}$ is irreducible modulo $\mfp_1$;
\item and $P_{\vec{a},0}(a_1/n),P_{\vec{a},0}(a_2),\dots,P_{\vec{a},0}(a_{n-1})$ are all distinct.
\end{enumerate} 

To do this, we'll need the following two lemmas. The following lemma is essentially Lemma 2.3 of \cite{ked} but has been suitably modified for our application.

\begin{lemma}
\label{lemmaone}
Let $\gamma \in \cO_K$ be such that $(\gamma)$ is a principal prime ideal coprime to $(n(n-1))$.  Then there exist infinitely many primes $\mfp_1$ modulo which the polynomial $R(x) = x^n - \gamma^{n-1}x + \gamma$ is irreducible and its derivative $R'(x) = nx^{n-1} - \gamma^{n-1}$ splits into linear factors.
\end{lemma}
\begin{proof}
The polynomial $R'(x)$ has splitting field $L = K(\zeta_{n-1},n^{1/(n-1)})$, in which $(\gamma)$ does not ramify because $(\gamma)$ does not divide $(n(n-1))$, where here $\zeta_{n-1}$ is an $(n-1)$th root of unity. Thus $R$ is an Eisenstein polynomial with respect to any prime above $(\gamma)$ in $L$; in particular, $R$ is irreducible over $L$. 

We first show that the Galois group of $R(x)$ over $L$ contains an $n$-cycle. The Chebotarev density theorem shows that the factorization types of the polynomial $R(x)$ modulo primes of $L$ are equidistributed over the Galois group of $R(x)$ over $L$. To show that $R(x)$ is irreducible modulo infinitely many primes of $L$, it is necessary to show that the Galois group of $R(x)$ over $L$ contains a full cycle. Let $\nu$ be a prime of $L$ above $\gamma$. Let $L_{\nu}$ be the completion of $L$ at $\nu$, and let $M_\nu/L_\nu$ be the extension given by $R(x)$. Since $R$ is Eisenstein, $L_\nu$ is totally ramified, and since $\nu \nmid n$, we have $L_\nu = K_\nu \left[\sqrt[n]{\pi}\right]$ for some uniformizer $\pi$ of $L_\nu$. The $n$-cycle in question is given by $\zeta_n^i \sqrt[n]{\pi} \rightarrow \zeta_n^{i+1}\sqrt[n]{\pi}$.

Now, the Chebotarev density theorem shows that there are infinitely many prime ideals of $L$ of degree $1$ over $K$ modulo which $R(x)$ is irreducible. The norm to $K$ of any such prime ideal is a prime ideal of the desired form.
\end{proof}

\begin{lemma}[Kedlaya, Lemma 2.4 \cite{ked}]
\label{lemmatwo}
For any field $F$ of characteristic zero, there exist $a_1,\dots,a_{n-1} \in F$ such that
\[
	P_{\vec{a},0}(a_1/n),P_{\vec{a},0}(a_2),\dots,P_{\vec{a},0}(a_{n-1})
\]
are distinct.
\end{lemma}

With these two lemmas in hand, we now proceed with our proof of Theorem~\ref{sievethm}. Let $\gamma$, $\mfp_1$, and $R(x)$ be as in \cref{lemmaone} and ensure that $\mfp_1 \nmid n! \mfa$. Apply \cref{lemmatwo} to obtain $a'_1,\dots,a'_{n-1}$ so that 
\[
	P_{\vec{a}',0}(a'_1/n),P_{\vec{a}',0}(a'_2),\dots,P_{\vec{a}',0}(a'_{n-1})
\]
have distinct values. Choose a prime ideal $\mfp_2 \nmid n!\mfa\mfp_1$ such the reductions have well-defined and distinct values modulo $\mfp_2$. Now use the Chinese remainder theorem to choose $a_1,\dots,a_{n-1} \in \cO_K$ subject to the following congruence conditions:
\begin{itemize}
\item $a_1$ is coprime to $(n)$ and is contained in $\mfa$ and the ideal generated by $n-1$;
\item the integers $a_2,\dots,a_{n-1}$ are contained in $(n!)$ and $a_i \equiv -1$ modulo $\mfa$ for $i = 2,\dots,n-1$;
\item the quantities $a_1/n,\dots,a_{n-1}$ are congruent to the roots of $R'(x)$ modulo $\mfp_1$;
\item and $a_i \equiv a'_i$ modulo $\mfp_2$;
\item for $i = 2,\dots,n-1$, consider the polynomial
\[
	F(x_2,\dots,x_{n-2}) = \prod_{i = 2}^{n-1} P_{(0,x_2,\dots,x_{n-1}),0}(x_i)
\]
as an element of the polynomial ring $\QQ[x_2,\dots,x_{n-1}]$. The degree of this polynomial is $n^{n-2}$. The denominators of the coefficients of $F$ are prime to $\mfa$, so we may reduce $F$ modulo $\mfa$ and consider it as an element of $\cO_K/\mfa[x_2,\dots,x_{n-1}]$. It is easy to check that the polynomial $F$ is nonzero modulo $\mfa$. By assumption, $\#(\cO_K/\mfa) > n^{n-2}$, so by the Schwartz-Zippel lemma, the vanishing locus of $F$ inside the affine space $\AA^{n-2}_{\cO_K/\mfa}$ is not the entirety of the affine space. Choose $a_2,\dots,a_{n-1}$ so that their reductions modulo $\mfa$ lie outside the vanishing locus of $F$.
\end{itemize}

We'll now show that this choice of $a_1,\dots,a_{n-1}$ satisfies conditions $(1)-(6)$.
\begin{enumerate}
\item Follows directly from the choice of $a_i$.
\item Observe that:
\[
	Q_{\vec{a}}(x) \equiv (nx - a_1)x^{n-2} = nx^{n-1} - a_1x^{n-2} \pmod{(n!)},
\]
so $P_{\vec{a},0}$ has integer coefficients.
\item Suppose that $\mfp \neq \mfa$ is a prime ideal dividing $(n!)$. Then for $b \in \cO_K$, we have $P_{\vec{a},0}(a_i)+ b \equiv b \pmod{\mfp}$ for $i = 2,\dots,n-1$ because $a_i \in \mfp$. Now suppose $\mfp \mid (n)$. Then because $Q_{\vec{a}}(x) \equiv nx^{n-1} - a_1x^{n-2} \pmod{(n!)}$, we have $P_{\vec{a},0} = x^n - \frac{a_1}{n-1}x^{n-1} \pmod{(n!)}$. So:
\begin{align*}
n^nP_{\vec{a},0}(a_1/n) + n^nb &\equiv n^{n}\bigg(\frac{a_1^n}{n^n} - \frac{a_1}{n-1}\frac{a_1^{n-1}}{n^{n-1}}\bigg) \pmod{\mfp} \\
&\equiv a_1^n - \frac{n}{n-1} a_1^n \pmod{\mfp} \\
&\equiv a_1^n \pmod{\mfp}.
\end{align*}
If $\mfp \nmid (n)$, then there exists $b \not\equiv 0 \pmod{\mfp}$ can be chosen such that 
\[
	n^nP_{\vec{a},0}(a_1/n) + n^nb \not\equiv 0 \pmod{\mfp}
\] 
Consequently, there exists $b$ such that $\Delta_{\vec{a},b}$ is not divisible by $\mfp$. Now suppose that $\mfp$ is a prime ideal coprime to $\mfa(n!)$, so $\lvert \cO_K/\mfp \rvert > n$. There are at least $\lvert \cO_K/\mfp \rvert - n$ choices of $b \in \cO_K/\mfp$ for which $\Delta_{\vec{a},b}$ is indivisible by $p$, because each linear factor rules out exactly one choice of $b$. 
\item Choose $b \in \mfa^2$ so that $P_{\vec{a},0}(a_1/n) + b \not \equiv 0 \pmod{\mfa^3}$. To conclude it suffices to show that $P_{\vec{a},0}(a_i) \not \equiv 0 \pmod{\mfa}$ for all $i = 2,\dots,n-1$; this follows from the fact that $a_2,\dots,a_{n-2}$ were chosen so that their reductions modulo $\mfa$ lie outside the vanishing locus of $F$.
\item Ensured by the fact that the quantities $a_1/n,\dots,a_{n-1}$ are congruent to the roots of $R'(x)$ modulo $\mfp_1$.
\item Ensured by the fact that $a_i \equiv a'_i$ modulo $\mfp_2$.
\end{enumerate}

Given such a choice of the $a_1,\dots,a_{n-1}$ above, let $T$ be the set of all $b \in \mfa^2$ such that $b \equiv b_1 \pmod{\mfp_1}$ amd $b \equiv b_2 \pmod{\mfa^3}$. For any $b \in T$, the polynomial $P_{\vec{a},b}(x) = x^n + f_1x^{n-1} + \dots + f_n$ is irreducible with integral coefficients such that $f_n \in \mfa^2$, $f_{n-1} \in \mfa$, $\mfp_1 \nmid \Delta_{\vec{a},b}$, and $\mfa^3 \nmid \Delta_{\vec{a},b}$. For $b \in K$, put $\lvert b \rvert = \sqrt{\sum_{i = 1}^{\deg(K)}\lvert \sigma_i(b)\rvert}$ where $\sigma_1,\dots,\sigma_{\deg(K)}$ are the embeddings of $K$ into $\CC$. The proof of \cref{sievethm} is concluded by the following lemma.

\begin{lemma}
Let $T$ be a set in $\cO_K$ defined by congruence conditions at a finite set of primes $S$. Let $c_1,\dots,c_k,d_1,\dots,d_k \in \cO_K$ and put $A(x) = \prod_{i = 1}^k(c_i x + d_i)$. Suppose that $A(x)$ is squarefree away from $S$. For $\mfp \neq \mfa$, let
\[
	a(\mfp) = \frac{\#\{b \in \cO_K/\mfp^2 \; \colon \; A(b) \not \equiv 0 \pmod{\mfp^2} \}}{\#(\cO_K/\mfp^2)}
\]
Then
\[
	\lim_{N \rightarrow \infty}\frac{\#\{b \in T \colon \lvert b\rvert \leq N \text{ and } A(b) \text{ is squarefree away from } S\}}{\#\{b \in T \colon \lvert b\rvert \leq N\}} = \prod_{\mfp \notin S}a(\mfp).
\]
If $a(\mfp) > 0$ for all $\mfp \notin S$, then $\prod_{\mfp \notin S}a(\mfp) > 0$.
\end{lemma}
\begin{proof}
If $a(\mfp) > 0$ for all $\mfp \notin S$, then there exists a $k$ such that:
\[
	\prod_{\mfp \notin S}a(\mfp) \geq \prod_{\mfp \notin S} \bigg (1 - \frac{k}{N_{K/\QQ}(\mfp)^2}\bigg) > 0,
\]
which proves the second claim. To prove the first claim it suffices to prove the following tail estimate:
\[
	\#\{b \in T \colon \lvert b \rvert \leq N, \exists \; \mfp \text{ s.t. } \Nm(\mfp) > \sqrt{N} \text{ and } A(b) \equiv 0 \pmod{\mfp^2}\} = o(N^{\deg(K)}).
\]
The tail estimate is bounded above by the sum of the following expressions:
\[
	\sum_{i = 1}^k \sum_{\Nm(\mfp) > \sqrt{N}}\#\{b \in T \colon \lvert b \rvert \leq N \text{ and } c_ib + d_i \equiv 0 \pmod{\mfp^2}\}
\]
\[
	\sum_{1 \leq i < j \leq k}^k \sum_{\Nm(\mfp) > \sqrt{N}}\#\{b \in T \colon \lvert b \rvert \leq N \text{ and } c_ib + d_i \equiv c_jb + d_j \equiv 0 \pmod{\mfp}\}
\]
Both sums are bounded above by $O(N^{\deg(K)-1/2})$, completing the proof.
\end{proof}


\begin{bibdiv}
\begin{biblist}

\bib{bsw}{article}{
	author = {Bhargava, Manjul},
	author = {Shankar, Arul},
	author = {Wang, Xiaoheng}
	title = {Geometry-of-numbers over global fields I: Prehomogeneous vector spaces},
	year={2015},
    eprint={1512.03035},
    archivePrefix={arXiv},
    primaryClass={math.NT},
}

\bib{kab}{article}{
	author = {Kable, Anthony C.},
	author = {Wright, David J.}
	title = {Uniform distribution of the Steinitz invariants of
quadratic and cubic extensions},
	journal={Compositio Math.},
	volume = {142},
	year={2006},
	pages={84--100},
}

\bib{ked}{article}{
  author = {Kedlaya, Kiran},
    title = {A construction of polynomials with squarefree discriminants},
    journal = {Proceedings of the American Mathematical Society},
    volume = {140},
    number = {9},
    pages = {3025-3033},
    year = {2012},
    month = {09},
    issn = {0002-9939},
    url = {http://hdl.handle.net/1721.1/80368},
    eprint = {https://www.ams.org/journals/proc/2012-140-09/S0002-9939-2012-11231-6/S0002-9939-2012-11231-6.pdf},
}

\bib{wood}{article}{
  author = {Wood, Melanie M.},
    title = {Rings and ideals parameterized by binary n-ic forms},
    journal = {Journal of the London Mathematical Society},
    volume = {83},
    number = {1},
    pages = {208-231},
    year = {2011},
    month = {01},
    issn = {0024-6107},
    doi = {10.1112/jlms/jdq074},
    url = {https://doi.org/10.1112/jlms/jdq074},
    eprint = {https://academic.oup.com/jlms/article-pdf/83/1/208/2434402/jdq074.pdf},
}

\end{biblist}
\end{bibdiv}
\end{document}